\documentclass[a4paper,reqno,12pt]{amsart}

\usepackage[utf8]{inputenc}
\usepackage[english]{babel}
\usepackage{amssymb}
\usepackage{amsfonts}
\usepackage{ mathrsfs }
\usepackage{multicol}
\usepackage{wrapfig}
\usepackage{stmaryrd}
\usepackage{graphicx}
\usepackage{fullpage}
\usepackage{amsmath}
\usepackage{amsthm}
\usepackage{amscd}
\usepackage{cite}
\usepackage{tikz}
\usepackage{verbatim}
    
\usepackage[colorlinks, citecolor=blue]{hyperref}

\DeclareMathOperator {\Ker}{ker}
\DeclareMathOperator {\spann}{span}

\DeclareMathOperator {\Cl}{Cl}
\DeclareMathOperator {\divisor}{div}
\DeclareMathOperator {\WDiv}{WDiv}
\DeclareMathOperator {\Spec}{Spec}

\DeclareMathOperator {\Aut}{Aut}

\DeclareMathOperator {\Reg}{Reg}
\DeclareMathOperator {\Ad}{Ad}
\DeclareMathOperator {\Ann}{Ann}
\DeclareMathOperator {\ad}{ad}

\def\A  {\mathbb A}
\def\N  {\mathbb N}
\def\Z  {\mathbb Z}
\def\C  {\mathbb C}

\def\Ga  {\mathbb G_a}
\def\Q  {\mathbb Q}
\def\K  {\mathbb K}
\def\P  {\mathbb P}
\def\epsilon{\varepsilon}

\let\emptyset\varnothing
\newtheorem{lmm}{Lemma}
\newtheorem{exmp}{Example}
\newtheorem{corollary}{Corollary}
\newtheorem{theorem}{Theorem}

\let\oldexmp\exmp
\renewcommand{\exmp}{\oldexmp\normalfont}

\newtheorem{oldstm}{Proposition}

\theoremstyle{definition}
\newtheorem{defn}{Definition}

\newtheorem{prob}{Problem}
\newtheorem{remark}{Remark}

\begin{document}\sloppy

\title{Additive actions on complete toric surfaces}
\author{Sergey Dzhunusov}
\address{Lomonosov Moscow State University, Faculty of Mechanics and Mathematics, Department
of Higher Algebra, Leninskie Gory 1, Moscow, 119991 Russia}
\email{dzhunusov398@gmail.com}
\renewcommand{\address}{{
    \footnotesize
    \textsc{Moscow}\par\nopagebreak
    \textit{E-mail} : \texttt{dzhunusov398@gmail.com}
  }}

\date{}
\thanks{The author was supported by RSF grant 19-11-00172.}
\begin{abstract}
  By an additive action on an algebraic variety $X$ we mean a regular effective action $\mathbb{G}_a^n\times X\to X$ with an open orbit of the commutative unipotent group~$\mathbb{G}_a^n$.
  In this paper, we give a classification of additive actions on complete toric surfaces. 

\end{abstract}


\maketitle
\section{Introduction}
Let $X$ be an irreducible algebraic variety of dimension $n$ over an algebraically closed field $\K$ of characteristic zero and $\Ga=(\K,+)$ be the additive group of the ground field.
Consider the commutative unipotent group $\Ga^n=\Ga\times\ldots\times\Ga$ ($n$~times).
By an additive action on $X$ we mean an effective regular action $\Ga^n\times X\to X$ with an open orbit.
In other words, an additive action on a complete variety $X$ allows to consider $X$ as a completion of the affine space $\A^n$ that is equivariant with respect to the group of parallel translations on $\A^n$.

The study of additive actions began with the work of Hassett and Tschinkel~\cite{HT}.
They introduced a correspondence between additive actions on the projective space $\P^n$  and local $(n+1)$-dimensional commutative associative algebras with unit; see also~\cite[Proposition~5.1]{KL} for a more general correspondence.
Hassett-Tshinkel's correspondence makes it possible to classify additive actions on $\P^n$ for $n \leq 5$; these are precisely the cases when the number of additive actions is finite.

The study of additive actions was originally motivated by problems of arithmetic geometry. Chambert-Loir and Tschinkel~\cite{CLT1, CLT2} gave asymptotic formulas for the number of rational points of bounded height on smooth (partial) equivariant compactifications of the vector group.

There is a number of results on additive actions
on flag varieties \cite{A1, Fe, FH, Dev},
singular del Pezzo surfaces \cite{DL},
Hirzebruch surfaces \cite{HT} and weighted projective planes \cite{ABZ}.

This work concerns the case of toric varieties.
The problem of classification of additive actions on toric varieties is raised in~\cite[Section~6]{AS}.

It is proved in~\cite{De} that $\Ga$-actions on a toric variety $X_{\Sigma}$ normalized by the acting torus~$T$ are in bijection with some elements $e \in M$, where $M$ is the character lattice of torus~$T$.
These vectors are called Demazure roots of the corresponding fan $\Sigma$.
Cox~\cite{Cox} observed that normalized $\Ga$-actions on a toric variety can be interpreted as certain $\Ga$-subgroups of automorphisms of the Cox ring $R(X)$ of the variety $X$.
In turn, such subgroups correspond to homogeneous locally nilpotent derivations of the Cox ring.

In~\cite{AR} all toric varieties admitting an additive action are described.
It turns out that if a complete toric variety $X$ admits an additive action, then it admits an additive action normalized by the acting torus.
Moreover, any two normalized additive actions on $X$ are isomorphic.

This work solves the problem of classification of additive actions for a complete toric surface.
It was known in particular cases of the projective plane~\cite[Proposition~3.2]{HT}, Hirzebruch surfaces~\cite[Proposition~5.5]{HT} and some weighted projective planes~\cite[Proposition~7]{ABZ} that these surfaces admit exactly two non-isomorphic additive actions, the normalized and the non-normalized ones.
In Theorem~\ref{main}, we prove that any complete toric surface admits at most two non-isomorphic additive actions and characterize the fans of surfaces that admit precisely two additive actions.

After presenting some preliminaries on toric varieties and Cox ring (Section~\ref{coxring}) and $\Ga$-actions and Demazure roots (Section~\ref{intrga}), we describe the results of~\cite{AR} (Section~\ref{intrbas}).
In Section~\ref{draa}, we prove some facts about Demazure roots of a toric variety admitting an additive action.
In Section~\ref{mainsection}, we formulate and prove the main theorem of this work.
In Section~\ref{example}, we give several explicit examples of additive actions on toric surfaces in Cox ring coordinates and discuss the further research.

The author is grateful to his supervisor Ivan Arzhantsev for posing the problem and permanent support and to Yulia Zaitseva for useful discussions and comments.

\section{Toric varieties and Cox rings}\label{coxring}
In this section we introduce basic notation of toric geometry,
see \cite{Fu, CLS} for details.

\begin{defn}
  A \emph{toric variety} is a normal variety~$X$ containing a torus~$T\simeq (\K^*)^n$ as a Zariski open subset such that the action of~$T$ on itself extends to an action of~$T$ on~$X$.
\end{defn}
Let~$M$ be the character lattice of~$T$ and~$N$ be the lattice of one-parameter subgroups of~$T$. Let~$\langle \cdot\,, \cdot\rangle: N \times M \to \Z$ be the natural pairing between the lattice N and the lattice M.
It extends to the pairing $\langle \cdot\,, \cdot\rangle_{\Q}: N_{\Q} \times M_{\Q} \to \Q$  between the vector spaces $N_{\Q}=N\otimes \Q$ and $M_{\Q}=M\otimes \Q$.
\begin{defn}
  A \emph{fan} $\Sigma$ in the vector space $N_{\Q}$ is a finite collection of strongly convex polyhedral cones $\sigma$ such that:
  \begin{enumerate}
  \item For all cones $\sigma \in \Sigma$, each face of $\sigma$ is also in $\Sigma$.
  \item For all cones $\sigma_1, \sigma_2 \in \Sigma$, the intersection $\sigma_1 \cap \sigma_2$ is a face of the cones $\sigma_1$ and $\sigma_2$.
  \end{enumerate}
\end{defn}

There is one-to-one correspondence between normal toric varieties $X$ and fans $\Sigma$ in the vector space $N_{\Q}$,
see \cite[Section~3.1]{CLS} for details.

\smallskip

Here we recall basic notions of the Cox construction, see~\cite[Chapter~1]{ADHL} for more details.
Let $X$ be a normal variety.
Suppose that the variety $X$ has free finitely generated divisor class group~$\Cl(X)$ and there are only constant invertible regular functions on~$X$.
Denote the group of Weil divisors on $X$ by $\WDiv(X)$ and consider a subgroup~$K \subseteq \WDiv(X)$ which maps onto $\Cl(X)$ isomorphically. The~ \emph{Cox~ring} of the variety $X$ is defined as
$$
R(X)=\bigoplus_{D\in K} H^0(X,D) = \bigoplus_{D\in K}\{f\in\K(X)^{\times}\mid \divisor(f)+D\geqslant0\}\cup\{0\}$$
and multiplication on homogeneous components coincides with multiplication in the field of rational functions $\K(X)$ and extends to the Cox ring $R(X)$ by linearity. It is easy to see that up to isomorphism the graded ring
$R(X)$ does not depend on the choice of the subgroup~$K$.

Suppose that the Cox ring $R(X)$ is finitely generated. Then ${\overline{X}:=\Spec R(X)}$ is a normal affine variety with an action of the torus $H_X := \Spec\K[\Cl(X)]$. There is an open $H_X$-invariant subset $\widehat{X}\subseteq \overline{X}$ such that the complement $\overline{X}\backslash\widehat{X}$ is of codimension at least two in $\overline{X}$,
there exists a good quotient $p_X\colon\widehat{X}\rightarrow\widehat{X}/\!/H_{X}$, and the quotient space $\widehat{X}/\!/H_{X}$ is isomorphic to $X$, see \cite[Construction~1.6.3.1]{ADHL}. Thus, we have the following diagram
$$
\begin{CD}
\widehat{X} @>{i}>> \overline{X}=\Spec R(X)\\
@VV{/\!/H_{X}}V  \\
X
\end{CD}
$$

It is proved in \cite{Cox} that if $X$ is toric, then $R(X)$ is a polynomial algebra
$\K[x_1,\ldots,x_m]$, where the variables $x_i$ correspond to $T$-invariant prime divisors $D_i$ on $X$ or, equivalently, to the rays $\rho_i$ of the corresponding fan $\Sigma$.
The $\Cl(X)$-grading on $R(X)$ is given by $\deg(x_i)=[D_i]$.
In this case, $\overline{X}$ is isomorphic to $\K^m$, and $\overline{X}\setminus\widehat{X}$ is a union of some coordinate subspaces in $\K^m$ of codimension at least two.
Denote by $\mathbb T$ the torus $(\K^{*})^m$ acting on the variety $\overline{X}$.
Each $w \in M$ gives a character $\chi^w : T \to \K^{*}$, and  hence $\chi^w$ is a rational function on $X$.
By \cite[Theorem 4.1.3]{CLS}, the function $\chi^w$ defines the principal divisor ${\divisor}(\chi^w) = - \sum_{\rho} \langle w,u_\rho\rangle D_\rho$.
Let us consider the map $M \longrightarrow \Z^{m} $ defined by~$w \mapsto  (\langle w,u_{\rho_1}\rangle, \ldots, \langle w,u_{\rho_m}\rangle)$, where $\rho_1, \ldots, \rho_m$ are one-dimensional cones of $\Sigma$.
We identify the group $\Z^m$ with the character lattice of the torus $(\K^*)^m$.
Thus, every element $w \in M$ corresponds to the character $\overline{\chi}^w$ of the torus $\mathbb T$.
Moreover, for any $w, w' \in M$  the equality $w = w'$ holds if and only if $\overline{\chi}^w =\overline{\chi}^{w'}$.

\section{Demazure roots and locally nilpotent derivations}\label{intrga}
Let $X_{\Sigma}$ be a toric variety of dimension $n$ and $\Sigma$ be the fan of the variety $X_{\Sigma}$.
Let $\Sigma(1)=\{\rho_1, \ldots, \rho_m\}$ in $N$ be the set of rays of the fan $\Sigma$ and $p_i$ be the primitive lattice vector on the ray $\rho_i$.

For any ray $\rho_i\in \Sigma(1)$ we consider the set $\mathfrak R_i$ of all vectors $e \in M$ such that
\begin{enumerate}
\item $\langle p_i, e\rangle=-1$ and $\langle p_j, e \rangle \geq 0$ for $j \neq i$, $1 \leq j \leq n$;
\item if  $\sigma$ is a cone of $\Sigma$ and $\langle v, e \rangle=0$ for all $v \in \sigma$, then the cone generated by $\sigma$ and $\rho_i$ is in $\Sigma$ as well.
\end{enumerate}

Elements of the set $\mathfrak R = \bigcup\limits_{i=1}^m \mathfrak R_i$ are called \emph{Demazure roots} of the fan $\Sigma$ (see \cite[Section~3.1]{De} or \cite[Section~3.4]{Oda}).
Let us divide the roots $\mathfrak R$ into two classes:
\begin{equation*}
\begin{array}{ccc}
\mathfrak S = \mathfrak R \cap -\mathfrak R &,\quad& \mathfrak U = \mathfrak R \setminus \mathfrak S.
\end{array}
\end{equation*}
Roots in $\mathfrak{S}$ and $\mathfrak{U}$ are called \emph{semisimple} and \emph{unipotent} respectively.

A derivation $\partial$ of an algebra $A$ is said to be \emph{locally nilpotent} (LND) if for every $f\in A$ there exists $k\in\N$ such that $\partial^k(f)=0$. For any LND $\partial$ on $A$ the map ${\varphi_{\partial}:\Ga\times A\rightarrow A}$, ${\varphi_{\partial}(s,f)=\exp(s\partial)(f)}$, defines a
structure of a rational $\Ga$-algebra on $A$. A derivation $\partial$ on a graded ring $A = \bigoplus\limits_{\omega \in K} A_{\omega}$ is said to be \emph{homogeneous} if it respects the~$K$-grading. If ${f,h\in A\backslash \Ker\partial}$ are homogeneous, then ${\partial(fh)=f\partial(h)+\partial(f)h}$ is homogeneous too and ${\deg\partial(f)-\deg f=\deg\partial(h)-\deg h}$. So any homogeneous derivation $\partial$ has a well-defined \emph{degree} given
as $\deg\partial=\deg\partial(f)-\deg f$ for any homogeneous $f\in A\backslash \Ker\partial$.

Every locally nilpotent derivation of degree zero on the Cox ring $R(X_{\Sigma})$ induces a regular action~$\Ga\times X_{\Sigma}\to X_{\Sigma}$. In fact, any regular $\Ga$-action on $X_{\Sigma}$ arises this way, see \cite[Section~4]{Cox} and \cite[Theorem~4.2.3.2]{ADHL}.
If a $\Ga$-action on a variety $X_{\Sigma}$ is normalized by the acting torus $T$, then the lifted
$\Ga$-action on $\overline{X}_{\Sigma}=\K^m$ is normalized by the diagonal torus $\overline{T} = (\K^{*})^m$. Conversely,
any $\Ga$-action on $\K^m$ normalized by the torus $(\K^{\times})^m$ and commuting
with the subtorus $H_{X_{\Sigma}}$ induces a $\Ga$-action on $X_{\Sigma}$. This shows that $\Ga$-actions on $X_{\Sigma}$
normalized by $T$ are in bijection with locally nilpotent derivations of the Cox ring
$\K[x_1,\ldots,x_m]$ that are homogeneous with respect to the standard grading by the lattice
$\Z^m$ and have degree zero with respect to the $\Cl(X_{\Sigma})$-grading.

For any element $e\in \mathfrak R_i$ we consider a locally nilpotent derivation $\prod\limits_{j \neq i}x_{j}^{\langle p_j, e \rangle}\frac{\partial}{\partial x_i}$ on the algebra $R(X_{\Sigma})$.
This derivation has degree zero with respect to the grading by the group~$\Cl(X_{\Sigma})$.
This way one obtains a bijection between the Demazure roots in $\mathfrak R$ and locally nilpotent derivations of degree zero on the ring $R(X_{\Sigma})$, which in turn are in bijection with  $\Ga$-actions on $X_{\Sigma}$ normalized by the acting torus.

\begin{lmm}\label{lndconj}
  Let $D_e$ be a homogeneous LND that corresponds to the Demazure root $e\in M$ and $t$ be an element of maximal torus $\mathbb{T}$.
  Then,
  $$tD_et^{-1} = \overline{\chi}^e(t)D_e.$$
\end{lmm}
\begin{proof}
  By definition, the derivation $D_e$ is equal to $\prod\limits_{j\neq i} x_j^{\langle p_j, e \rangle}\frac{\partial}{\partial x_i}$.
  Let us consider the image $tD_et^{-1}(x_i)$ of an element $x_i$. It is equal to $t_i^{-1}\prod\limits_{j\neq i} t_j^{\langle p_j, e \rangle}\prod\limits_{j\neq i} x_j^{\langle p_j, e \rangle}$.
  Thus, we get
  $$tD_et^{-1}=t_i^{-1}\prod\limits_{j\neq i} t_j^{\langle p_j, e \rangle} D_e = \prod\limits_{j=1}^m t_j^{\langle p_j, e \rangle}D_e = \overline{\chi}^e(t)D_e.$$
\end{proof}

\section{Complete toric varieties admitting an additive action}\label{intrbas}
In this section, we shortly present the results of \cite{AR}.
Let $X_{\Sigma}$ be a toric variety of dimension $n$ admitting an additive action and $\Sigma$ be the fan of the variety $X_{\Sigma}$.
Denote primitive vectors on the rays of the fan $\Sigma$ by $p_i$, where $1 \leq i \leq m$.

\begin{defn} \label{deff}
A set $e_1,\ldots,e_n$ of Demazure roots of a fan $\Sigma$ of dimension $n$ is called a {\it complete collection} if $\langle p_i,e_j\rangle=-\delta_{ij}$ for all $1\le i,j\le n$ for some ordering of $p_1, \ldots, p_m$.
\end{defn}
An additive action on a toric variety $X_{\Sigma}$ is said to be \emph{normalized} if the image of the group $\Ga^n$ in $\Aut(X_{\Sigma})$ is normalized by the acting torus.
\begin{theorem}{\cite[Theorem 1]{AR}} \label{cc}
  Let $X_{\Sigma}$ be a toric variety. Then normalized additive actions on $X_{\Sigma}$ normalized are in bijection with complete collections of Demazure roots of the fan $\Sigma$.
\end{theorem}

\begin{corollary}
A toric variety $X_{\Sigma}$ admits a normalized additive action if and only if there is a complete collection of Demazure roots of the fan $\Sigma$.
\end{corollary}

\begin{theorem}{\cite[Theorem 3]{AR}}\label{3con}
Let $X_{\Sigma}$ be a complete toric variety with an acting torus~$T$. The following conditions are equivalent.
\begin{itemize}
\item[(1)]
There exists an additive action on $X_{\Sigma}$.
\item[(2)]
  There exists a normalized additive action on $X_{\Sigma}$.
\end{itemize}
\end{theorem}
Here we prove a proposition that will be used below.
\begin{defn}
  The \emph{negative octant} of the rational vector space $V$ with respect to a basis $f_1, \ldots, f_n$ is the cone~$\left\{\sum\limits_{i=1}^n \lambda_i f_i \mid \lambda_i \leq 0\right\} \subset V$.
\end{defn}
\begin{stm}\label{ort}
  Let $X_{\Sigma}$ be a complete toric variety.
  The following statements are equivalent.
  \begin{enumerate}
  \item There is a complete collection of Demazure roots of the fan $\Sigma$.
  \item We can order rays of the fan $\Sigma$ in such a way that the primitive vectors on the first $n$ rays form a basis of the lattice $N$ and the remaining rays lie in the negative octant with respect to this basis.
  \item There exists an additive action on $X_{\Sigma}$.
  \end{enumerate}
\end{stm}
\begin{proof}
  Let us prove implication $(1) \Rightarrow (2)$.
  Assume  that the vectors $p_1, \ldots, p_n$ are linearly dependent, i.e. there exists a non-trivial linear relation~${\alpha_1p_1+\ldots+\alpha_np_n=0}$.
  Then we get $-\alpha_i = {\langle \alpha_1p_1+\ldots+ \alpha_np_n, e_i \rangle}=0$ for all $1 \leq i \leq n$, a contradiction.
  Consider an arbitrary vector $v=\sum_{i=1}^n\nu_ip_i$ of the lattice $N$.
  By definition of a complete collection, we get $\langle v, e_i \rangle = -\nu_i \in \Z$.
  Therefore, the vectors $p_1, \ldots, p_n$ form the basis of the lattice $N$.
  
  All other vectors $p_j,\, j>m,$ are equal to $-\sum_{l=1}^n\alpha_{jl}p_l$ for some integer $\alpha_{jl}$.
  By definition of a Demazure root, we obtain
  $$0 \leq \langle p_j, e_i \rangle = \sum \alpha_{jl}\delta_{li}=\alpha_{ji}.$$

  The converse implication is straightforward.

  Equivalence $(1) \Leftrightarrow (3)$ follows from Theorems~\ref{cc} and~\ref{3con}.
\end{proof}

\section{Demazure roots of a variety admitting an additive action}\label{draa}
Let $X_{\Sigma}$ be a complete toric variety of dimension $n$ admitting an additive action and $\Sigma$ be the fan of the variety $X_{\Sigma}$.
Denote primitive vectors on the rays of the fan $\Sigma$ by $p_i$, where $1 \leq i \leq m$.

From Proposition~\ref{ort} it follows that we can order $p_i$ in such a way that the first $n$ vectors form a basis of the lattice $N$ and the remaining vectors $p_j$ $(n < j \leq m)$ are equal to~$\sum_{i=1}^n-\alpha_{ji} p_i$ for some non-negative integers $\alpha_{ji}$.

Let us denote the dual basis of the basis $p_1, \ldots, p_n$ by $p_1^*,\,\ldots, p_n^*$.
\begin{lmm}\label{firstroot}
  Consider $1\leq i\leq n$.
  The set ${\mathfrak R}_i$ is a subset of the set ${-p_i^* + \sum\limits_{l=1,l\neq i}^n\Z_{\geq 0}p_j^*}$ and the vector $-p_i^*$ is a Demazure root from the set~${\mathfrak R}_i$.

 \end{lmm}
\begin{proof}
  Let $e=\sum\limits_{i=1}^n\epsilon_ip_i^*$ be a Demazure root from ${\mathfrak R}_i$.
  By the definition, the Demazure roots from $\mathfrak R_i$ are defined by the following equations:
  \begin{equation}
    %\epsilon_i = -1, \epsilon_l \geq 0,  &l \leq n,\, l \neq i, a_{ji}-\sum\limits_{\substack{l=1\\ l\neq i}}^n \epsilon_la_{jl} \geq 0, n < j \leq m
      \begin{array}{cccc}
        \epsilon_i = -1&\\
        \epsilon_l \geq 0,  &l \leq n,\, l \neq i\\
        \alpha_{ji}-\sum\limits_{\substack{l=1\\ l\neq i}}^n \epsilon_l\alpha_{jl} \geq 0,& n < j \leq m
      \end{array}
    \end{equation}
    It is clear that all possible solutions lie in the set ${-p_i^* + \sum\limits_{\substack{l=1\\l\neq i}}^n\Z_{\geq 0}p_l^*}$, and the vector $-p_i^*$ satisfies them.
\end{proof}

Consider the set $\Reg(\mathfrak S) =\{u \in N : \langle u, e \rangle \neq 0 \text{ for all } e \in \mathfrak S\}$.
Any element $u$ from $\Reg(\mathfrak S)$ divides the set of semisimple roots $\mathfrak S$ into two classes as follows:  $${\mathfrak S_{u}^+ =\{e \in \mathfrak S: \langle u, e\rangle > 0\}},\quad {\mathfrak S_{u}^- =\{e \in \mathfrak S: \langle u, e\rangle < 0\}}.$$
At this point, any element of ${\mathfrak S_{u}}^+$ is called \emph{positive} and any element of $\mathfrak S_{u}^-$ is called \emph{negative}.

\begin{stm}\label{selective}
  Let $X_{\Sigma}$ be a complete toric variety admitting an additive action, and ${\mathfrak R = \bigcup_{i=1}^{m} \mathfrak R_i}$ be the set of its Demazure roots.
  Then
  \begin{enumerate}
  \item any element $e \in  \mathfrak R_j, j > n$, is equal to $p_{i_j}^*$ for some $1 \leq i_j\leq n$;
  \item all unipotent Demazure roots lie in the set $\bigcup_{i=1}^{n}\mathfrak R_i$;
  \item there exists a vector $u\in \Reg(\mathfrak S)$ such that $\mathfrak S_{u}^+ \subset \mathfrak \bigcup_{i=1}^n \mathfrak R_i$.
  \end{enumerate}
\end{stm}
\begin{proof}
  We start with the first statement.
  Consider a root~${e =\sum\limits_{i=1}^n \epsilon_i p_i^* \in \mathfrak R_j}$, where $j>n$.
  By definition of Demazure roots, we have $-\langle p_j, e\rangle=\sum\limits_{i=1}^n\alpha_{ji}\epsilon_i=1$ and $\epsilon_i \geq 0$ for all $1\leq i\leq n$.
  Consider the set ${I_j = \{i : \alpha_{ji} > 0\}}$.
  Then there exists $s \in I_j$ such that $\epsilon_s=1$ and for all $l \in I_j \setminus \{s\}$ the equality $\epsilon_l=0$ holds.
  Since $X_{\Sigma}$ is complete, there is no half-space with all vectors $p_i$ inside.
  Hence, for all $l \in \{1, \ldots, n\} \setminus I_j$ there exists  $r > n$ such that $\alpha_{rl} > 0$.
  Since $\langle p_r, e\rangle = -\sum\limits_{i=1}^n\alpha_{ri}\epsilon_i \geq 0$, we have $\epsilon_l=0$.
  This implies $e=p_s^*$.
  The first statement is proved.

  Let us prove the second statement.
  As above, consider the root ${e = p_{i_j}^* \in \mathfrak R_j}$, $j>n$.
  From  the first statement of Proposition~\ref{selective} and Lemma~\ref{firstroot} it follows that the element $-e$ is a root and lies in $\mathfrak R_{i_j}$ for some $i_j$.
  This means that the root $e$ is semisimple.
  Hence, all unipotent roots lie in the set $\bigcup_{i=1}^n \mathfrak R_i$.

  To prove (3), we should find a vector $u$ from the set $\Reg(\mathfrak S)$ such that  the set $\bigcup_{j=n+1}^m\mathfrak R_j$ contains only negative roots.
  Consider the vector $u_0 = -\sum\limits_{i=1}^n p_i$.
  For every root $e\in \bigcup_{j=n+1}^m\mathfrak R_j$, we get the  inequality $\langle u_0, e \rangle = -1 < 0$.
  We can add a small rational vector $\Delta u = \frac{1}{Q}\Delta u' \in N_{\Q}$, where $\Delta u' \in N$ and $Q$ is a positive integer such that the inequality~$\langle u_0+\Delta u, e\rangle_{\Q} <0$ holds for all roots $e \in \bigcup_{i=n+1}^m \mathfrak R_i$.
  So, we have $Q(u_0+\Delta u) \in \Reg(\mathfrak S)$, and we obtain the required vector~$u:=Q(u_0+\Delta u)$.
\end{proof}

\section{Main results}\label{mainsection}
\begin{wrapfigure}{i}{0.22\textwidth}
\begin{picture}(90,90)
  \put(50,50){\vector(1,0){40}}
  \put(50,50){\vector(0,1){40}}
  \qbezier[40](0,0)(25,25)(50,50)
  \put(50,50){\line(0,-1){45}}
  \put(85,55){$p_1$}
  \put(55,85){$p_2$}
  \put(50,50){\line(-1,0){45}}
  \put(50,50){\vector(-3,-1){50}}
  \put(50,50){\vector(-1,-3){20}}
  \put(17,33){$A_{I}$}
  \put(32,18){$A_{II}$}
\end{picture}
\end{wrapfigure}

We consider a complete toric surface $X_{\Sigma}$ with the fan $\Sigma$ admitting an additive action.
Denote primitive vectors of the rays of the fan~$\Sigma$ by $p_1, \ldots, p_m$.
We assume that $p_1, p_2$ is the standard basis of $N_{\Q}$. 

\begin{defn}\label{defwide}
  Let us call a fan $\Sigma$ \emph{wide} if it satisfies one of the following equivalent conditions:
\begin{enumerate}
  \item There exist $2 < i_1, i_2 \leq m$ such that~${a_{i_11} > a_{i_12}}$ and ${a_{i_21} < a_{i_22}}$;
  \item $\mathfrak R_1 = \{-p_1^*\}$ and $\mathfrak R_2=\{-p_2^*\}$.
  \end{enumerate}
\end{defn}
\begin{proof}[Proof of Equivalence.]
  From the definition of Demazure roots it follows that
$$\mathfrak{R}_1 = \left\{(-1, k) : 0 \leq k \leq \min_{j>2}\left(\frac{\alpha_{j1}}{\alpha_{j2}}\right)\right\}, \quad \mathfrak{R}_2 = \left\{(k, -1) : 0 \leq k \leq \min_{j>2}\left(\frac{\alpha_{j2}}{\alpha_{j1}}\right)\right\}.$$

  From this it follows that $|\mathfrak{R}_1| =\left\lfloor \min\limits_{j>2}\left(\dfrac{\alpha_{j1}}{\alpha_{j2}}\right) \right\rfloor+ 1$, $|\mathfrak{R}_2| = \left\lfloor\min\limits_{j>2}\left(\dfrac{\alpha_{j2}}{\alpha_{j1}}\right)\right\rfloor + 1$.
  This implies the equivalence.
\end{proof}

Let us consider two areas in $N_{\Q}$: $${A_{I} = \{(x, y) \in M_{\Q} : x \leq 0, y \leq 0, x < y\}},$$
$${A_{II} = \{(x, y) \in M_{\Q} : x \leq 0, y \leq 0, x > y\}}.$$
The first condition from the definition of a wide fan means that there is a ray of $\Sigma$ in the area $A_{I}$ and there is a ray in the area $A_{II}$.

\smallskip

Now we are ready to formulate the main theorem.
\begin{theorem}\label{main}
Let $X_{\Sigma}$ be a complete toric surface  admitting an additive action.
  Then there is only one additive action on $X_{\Sigma}$ if and only if the fan $\Sigma$ is wide; otherwise there exist two non-isomorphic additive actions, one is normalized and the other is not.
\end{theorem}

\begin{proof}[Proof of Theorem~\ref{main}]
  We are going to classify additive actions on $X_{\Sigma}$ by describing two-dimensional subgroups of a maximal unipotent subgroup $U$ of the automorphism group $\Aut(X_{\Sigma})$ up to conjugation in $\Aut(X_{\Sigma})$.
  
  Fix a vector $u \in \Reg(\mathfrak S)$ that satisfies assertion~(3) of Proposition \ref{selective}.
  Hereafter we write $\mathfrak S^+$ instead of $\mathfrak S^+_{u}$.
  Denote the set $\mathfrak S^+ \cup \mathfrak U$ by $\mathfrak R^+$.
  From Proposition~\ref{selective} it follows that $\mathfrak R^+$ lies in the set $\bigcup_{i=1}^n\mathfrak R_i$.
All the one-parameter subgroups of roots from $\mathfrak R^+$ generate the maximal unipotent subgroup~$U$ in the group~$\Aut(X_{\Sigma})$, see~\cite[Proposition 4.3]{Cox}.
Denote the set ${\mathfrak R^+ \cap \mathfrak R_i}$ by $\mathfrak R^+_i$.

\begin{lmm}
  There exists $i\in \{1, 2\}$ such that $|\mathfrak R^+_i|=1$.
  Moreover, $\max_{i=1,2}|\mathfrak R_i^+|=\max_{i=1,2}|\mathfrak R_i|$.
\end{lmm}
\begin{proof}
  From the definition of Demazure roots it follows that
    $$\mathfrak{R}_1 = \left\{(-1, k) : 0 \leq k \leq \min_{j>2}\left(\frac{\alpha_{j1}}{\alpha_{j2}}\right)\right\}, \quad \mathfrak{R}_2 = \left\{(k, -1) : 0 \leq k \leq \min_{j>2}\left(\frac{\alpha_{j2}}{\alpha_{j1}}\right)\right\}.$$
  
    We have $|\mathfrak R_1|>1, |\mathfrak R_2|>1$ simultaneously if and only if
  \begin{gather*}
    {\mathfrak R}_1 = \{(-1, 0), (-1, 1) \}\\
    {\mathfrak R}_2 = \{(0, -1), (1, -1)\}.
  \end{gather*}
  Since the roots $(-1, 1), (1, -1)$ are opposite to each other, only one of them can lie in~$\mathfrak R^+$.

  \smallskip

  Only the roots $(-1, 1), (1, -1)$ can lie in the set $(\mathfrak R_1 \cap -\mathfrak R_2) \cup (\mathfrak R_2 \cap -\mathfrak R_1)$.
  Thus, we have $|\mathfrak R_1^+|=1$, $\mathfrak R_2^+=\mathfrak R_2$ or $|\mathfrak R_2^+|=1$, $\mathfrak R_1^+=\mathfrak R_1$.
\end{proof}
Without loss of generality, it can be assumed that $|\mathfrak R^+_1|=1$.
Denote the cardinality of the set $\mathfrak R^+_2$ by $d+1$.
By Definition~\ref{defwide} the fan is wide if and only if $d$ is equal to $0$.
In there term we have $\mathfrak R_1^+ = \{(-1, 0)\}$ and $\mathfrak R_2^+ = \{(k, -1) : 0 \leq k \leq d\}$.
Denote LND that corresponds to the root $(-1, 0) \in \mathfrak R_1^+$ by $\delta$, and LNDs that correspond to roots  $(k, -1) \in \mathfrak R^+_2, 0 \leq k \leq d$ by $\partial_k$.
\begin{lmm}\label{Lie}
  The following equations hold:
  $$
    [\delta, \partial_k] = k \partial_{k-1},
    \quad\quad
    [\partial_k, \partial_{k'}]=0.
  $$
\end{lmm}
\begin{proof}
  In this proof we use notation introduced in Section~\ref{coxring}.
  The correspondence between Demazure roots and LNDs implies:
  
  $$
    \delta=\prod\limits_{j=3}^m x_j^{\alpha_{j1}}\frac{\partial}{\partial x_1}, \quad \partial_k=x_1^k\prod\limits_{j=3}^mx_j^{\alpha_{j2}-k\alpha_{j1}}\frac{\partial}{\partial x_2}\\
  $$
  It can be easily checked that the derivations $\partial_k$ commute with each other.
  Moreover, direct computations show that the commutator  $[\delta, \partial_k]$ is equal to the derivation $k\partial_{k-1}$.
\end{proof}
    
Let us find all commutative subgroups in the group~$U$ that correspond to additive actions.
Such groups are in bijection with some pairs $(D_1, D_2)$ of commuting LNDs.
Note that not every pair of commuting LNDs corresponds to an additive action.
\begin{lmm}\label{sumform}
  In the above terms there is an invertible linear operator $\phi$ on the vector space $\langle D_1, D_2\rangle$ that sends the derivations $D_1, D_2$ to
  \begin{equation}\label{sum}
    \begin{cases}
      \phi(D_1)=\delta + \sum\limits_{k=0}^d\mu_k\partial_k\\
      \phi(D_2)=\partial_0
    \end{cases}
    ,\quad\quad \mu_k \in \K
  \end{equation} 
\end{lmm}
\begin{proof}
  Every pair of derivations has the form $D_1 = \lambda^{(1)}\delta+\sum \mu_k^{(1)}\partial_k$ and $ D_2 = \lambda^{(2)}\delta+\sum \mu_k^{(2)}\partial_k$.
  If $\lambda^{(1)}=\lambda^{(2)}=0$ then dimension of the orbit in the total space $X_{\Sigma}$ is less than~2 and the orbit can not become open after the factorization $\widehat{X_{\Sigma}} \to X_{\Sigma}$.
  Thus, without loss of generality we can assume that $\lambda^{(1)}\neq 0$.
  We can convert derivations $D_1, D_2$ to the  form $\delta+\sum \mu_k^{(1)}\partial_k, \sum \mu_k^{(2)}\partial_k$.
  From Lemma~\ref{Lie} it follows that the derivations $D_1, D_2$  commute if and only if~$\mu_k^{(2)}=0$ for~$k>0$.
  Thus, we can convert derivations $D_1, D_2$ to the form $\delta+\sum \mu_k^{(1)}\partial_k, \mu_0^{(2)}\partial_0$, with $\mu_0^{(2)}\neq 0$.
  We can assume that $\mu_0^{(2)}=1$.
\end{proof}
\begin{lmm}\label{cor}
  Every pair of derivations of form \eqref{sum} corresponds to an additive action.
\end{lmm}
\begin{proof}
  Let us consider the $\Ga^2$-action corresponding to the LNDs $D_1, D_2$.
  We  prove that the group~$\Ga^2 \times H_{X_{\Sigma}}$ acts in the total space $\K^m$ with an open orbit.
  By construction, the group~$\Ga^2$ changes exactly two of the coordinates $x_1,\ldots,x_m$, while the weights of the remaining $m-2$ coordinates with respect to the $\Cl(X)$-grading form a basis of the lattice of characters of the torus $H_X$.
  From this it follows that  there exists a point $p \in \K^m$ with trivial stabilizer.
  Due to $\dim(\Ga^2 \times H_{X_{\Sigma}}) = m$ we get that the orbit of the point $p$ is open.
\end{proof}
Hereafter, we suppose that $D_1, D_2$ have form \eqref{sum}.
From Lemma~\ref{sumform} it follows that if $d=0$, then derivations $D_1, D_2$ can be converted to $\delta, \partial_0$ respectively.
Such LNDs correspond to a normalized additive action and every additive action is isomorphic to this action.

\smallskip
Hereafter, we assume that $d\neq 0$.
\begin{lmm}\label{conj}
  There exists an automorphism $\psi \in \Aut(R(X_{\Sigma}))$  that conjugates $D_1, D_2$ to the form
  \begin{equation}\label{final}
    \begin{cases}
      \psi(D_1)=\delta + \mu_d\partial_d\\
      \psi(D_2)=\partial_0
    \end{cases}
  \end{equation}
\end{lmm}
\begin{proof}
  We are going to find numbers $\eta_k \in \K$ such that the automorphism $\psi = \exp(\delta + \sum\limits_{k=1}^d\eta_k\partial_k)$ is the desired one.
  
  The automorphism $\psi$ conjugates LNDs $D_1, D_2$ to the form
  \begin{gather*}
    \exp(\delta +\sum\limits_k\eta_k \partial_{k})D_1\exp(-\delta -\sum\limits_k\eta_k\partial_k)=\\
    =\Ad\left(\exp\left(\delta +\sum\limits_k\eta_k \partial_{k}\right)\right)D_1 = \exp\left(\ad\left(\delta +\sum\limits_k\eta_k \partial_k\right)\right)D_1 =\\
    =D_1 + \sum\limits_{l=1}^{\infty}\frac{\ad\left(\delta +\sum\limits_k\eta_k \partial_{k}\right)^l}{l!}D_1
    =\delta + \sum\limits_{k=0}^d\left(\mu_k + \sum\limits_{l=1}^{d-k}\frac{(k+l)!}{k!}(-\mu_{k+l}+\eta_ {k+l})\right)\partial_k;\\\
    \exp(\delta +\sum\limits_k\eta_k \partial_{k})D_2\exp(-\delta -\sum\limits_k\eta_k\partial_k)=D_2.
  \end{gather*}
  Here, we get the system of linear equations
  $$\mu_k + \sum\limits_{l=1}^{d-k}\dfrac{(k+l)!}{k!}(-\mu_{k+l}+\eta_{k+l})=0,\; 0 \leq k \leq d-1,$$
  in variables $\eta_1, \ldots, \eta_d$.
  This system has a unique solution as an upper triangular system and it is the solution  we are looking for.
\end{proof}
Hereafter, we suppose that $D_1, D_2$ have form~\eqref{final}.
Thus, we have a family of additive actions parameterized by the number $\mu_d$:
\begin{equation}
  \arraycolsep=0.3pt
  \begin{array}{ccccl}
    x_1 &\to& \exp(s_1D_1+s_2D_2)x_1&=&x_1 +s_1\prod\limits_{j=3}^m x_j^{\alpha_{j1}}\\
    x_2 &\to& \exp(s_1D_1+s_2D_2)x_2&=&x_2 + (s_2 + \frac{\mu_ds_1^d}{d!})\prod\limits_{j=3}^mx_j^{\alpha_{j2}}+\sum\limits_{k=1}^{d}\frac{\mu_ds_1^{d-k}}{k!}x_1^k\prod\limits_{j=3}^mx_j^{\alpha_{j2}-k\alpha_{j1}}\\
  \end{array}
\end{equation}
Note that every action corresponding to the pair of LNDs of form~\eqref{final} acts on $x_j, 3\leq j \leq m$ identically.
\begin{lmm}\label{nnaisom}
  All additive actions with $\mu_d \neq 0$ are non-normalized and isomorphic to each other.
\end{lmm}
\begin{proof}
  We conjugate the pair of LNDs that have form~\eqref{final} by an element $t$ of the maximal torus $\mathbb{T} = (\K^*)^m$.
  Using Lemma~\ref{lndconj} we obtain
  \begin{gather*}
    tD_1t^{-1}=\overline{\chi}^{(-1,0)}(t)\delta + \mu_d\overline{\chi}^{(d,-1)}(t)\partial_d \\
    tD_2t^{-1}=\overline{\chi}^{(0,-1)}(t)\partial_0
  \end{gather*}
  Since $\overline{\chi}^{(-1,0)} \neq \overline{\chi}^{(d,-1)}$  we can conjugate an additive action with $\mu_d \neq 0$ to the additive action with $\mu_d=1$.
\end{proof}

From the last lemma it follows that there are two classes of additive actions.
The first one ($\mu_d=0$) is a normalized additive action:
\begin{equation}\label{NA}
  \arraycolsep=0.3pt
  \begin{array}{ccccl}
    x_1 &\to&x_1 +s_1\prod\limits_{j=3}^m x_j^{\alpha_{j1}}\\
    x_2 &\to&x_2 + s_2\prod\limits_{j=3}^mx_j^{\alpha_{j2}}.\\
  \end{array}
\end{equation}\label{NNA}
The second is a non-normalized additive action:
\begin{equation}
  \arraycolsep=0.3pt
  \begin{array}{ccl}
    x_1 &\to&x_1 +s_1\prod\limits_{j=3}^m x_j^{\alpha_{j1}}\\
    x_2 &\to&x_2 + (s_2 + \frac{s_1^d}{d!})\prod\limits_{j=3}^mx_j^{\alpha_{j2}}+\sum\limits_{k=1}^{d}\frac{s_1^{d-k}}{k!}x_1^k\prod\limits_{j=3}^mx_j^{\alpha_{j2}-k\alpha_{j1}}.\\
  \end{array}
\end{equation}

\begin{lmm}\label{notisom}
  Actions~\eqref{NA} and~\eqref{NNA} are not isomorphic.
\end{lmm}
\begin{proof}
  Let us consider the homogeneous component of $\K[\overline{X}]$ containing $x_2$:
  $$C = \langle x_2 \rangle \oplus \spann\{x_1^{k}\prod_{j=3}^m x_j^{\alpha_{j2}-k\alpha_{j1}} : 0\leq k \leq d\}.$$

  We consider the space $V =  \{s_1D_1+s_2 D_2 : s_1, s_2 \in \K\}$ and its subspace 
  $$\Ann_Vf = \{v \in V : vf = 0\}, \, f \in C.$$
  Let $f = \lambda x_2 + \sum\limits_{k=0}^d \lambda_k x_1^{k}\prod\limits_{j=3}^m x_j^{\alpha_{j2}-k\alpha_{j1}}$ be an arbitrary non-zero element of $C$.

  In the case of normalized action $(s_1D_1 + s_2D_2)f$ is equal to
  $$s_2\lambda \prod_{j=3}^m x_j^{\alpha_{j2}} + s_1 \sum_{k=1}^d\lambda_kkx_1^{k-1}\prod_{j=3}^m x_j^{\alpha_{j2}-(k-1)\alpha_{j1}}.$$
  Elements of $\Ann_V f$ are defined by the following equations:
    \begin{equation}
      \begin{array}{cc}
        \lambda s_2 +\lambda_1s_1 = 0&\\
        \lambda_ks_1 = 0,& 2 \leq k \leq d
      \end{array}
    \end{equation}
    
    The collection of subspaces $\Ann_V f$, where $f \in C\setminus \{0\}$, contains a family of lines~${\{s_1D_1+ s_2D_2 : \lambda_1 s_1+ \lambda s_2=0\}}, (\lambda : \lambda_1) \in \P^2$.

    In the case of non-normalized action $(s_1D_1 + s_2D_2)f$ is equal to
  $$s_2\lambda \prod_{j=3}^m x_j^{\alpha_{j2}} +s_1 \lambda x_1^{d}\prod_{j=3}^m x_j^{\alpha_{j2}-d\alpha_{j1}}+ s_1 \sum_{k=1}^d\lambda_kkx_1^{k-1}\prod_{j=3}^m x_j^{\alpha_{j2}-(k-1)\alpha_{j1}}.$$
  Elements of $\Ann_V f$ are defined by the following equations:
    \begin{equation}
      \begin{array}{cc}
        \lambda s_2 + \lambda_1s_1 = 0&\\
        \lambda_ks_1 = 0 ,&2 \leq k \leq d\\
        \lambda s_1=0&\\
      \end{array}
    \end{equation}
    The subspace $\Ann_V f$ for $f \in \C \setminus \{0\}$ can be either $\K D_2$ or $0$.

    Hence, actions~\eqref{NA} and~\eqref{NNA} are not isomorphic.
  \end{proof}
  \begin{remark}
    The idea of this proof is taken from the proof \cite[Theorem~1]{ABZ}.
  \end{remark}

  \smallskip
  In the case of a wide fan Theorem~\ref{main} follows from Lemmas \ref{sumform} and \ref{cor}.
  In the case of a non-wide fan we obtain the assertion  from Lemmas \ref{cor}-\ref{notisom}.
Theorem~\ref{main} is proved.
\end{proof}
\section{Examples and problems}\label{example}
In this section, we describe some examples illustrating  Theorem~\ref{main}.

\begin{exmp} Let us consider the surface $\P^1\times \P^1$.
  Its fan is wide and there is only one additive action up to isomorphism.

  \begin{multicols}{3}
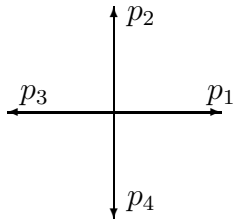
    \begin{picture}(90,90)
      \put(50,50){\vector(1,0){40}}
      \put(50,50){\vector(0,1){40}}
      \put(50,50){\vector(-1,0){40}}
      \put(50,50){\vector(0,-1){40}}      
      \put(85,55){$p_1$}
      \put(55,85){$p_2$}
      \put(15,55){$p_3$}
      \put(55,15){$p_4$}

    \end{picture}

    \columnbreak

    $\arraycolsep=0.0pt\begin{array}{c}
  \mathfrak R_1 = \{(-1, 0)\}\\
  \mathfrak R_2 = \{(0, -1)\}\\
  \mathfrak R_3 = \{(1, 0)\}\\
  \mathfrak R_4 = \{(0, 1)\}\\
  \mathfrak R^+=\{(-1, 0), (0, -1)\}
 \end{array}$

\columnbreak

\begin{center}
  Normalized action:
  
  $\arraycolsep=0.0pt  \begin{array}{ccc}
    x_1 &\to& x_1 + s_1 x_3\\
    x_2 &\to& x_2 + s_2 x_4\\
    x_3 &\to& x_3\\
    x_4 &\to& x_4\\
  \end{array}  $

$(s_1, s_2) \in \Ga^2$
  \end{center}
  
\end{multicols}

\end{exmp}
\newpage
\begin{exmp} Let us consider the surface corresponding to the following fan with ${p_3=-p_1-2p_2, p_4=-2p_1-p_2}$.
  Its fan is wide and there is only one additive action up to isomorphism.

  \begin{multicols}{3}
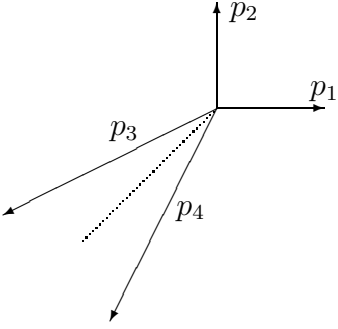
    \begin{picture}(90,110)
      \put(50,70){\vector(1,0){40}}
      \put(50,70){\vector(0,1){40}}
      \put(50,70){\vector(-1,-2){40}}
      \put(50,70){\vector(-2,-1){80}}
      \qbezier[40](0,20)(25,45)(50,70)
      \put(85,75){$p_1$}
      \put(55,105){$p_2$}
      \put(10,60){$p_3$}
      \put(35,30){$p_4$}

    \end{picture}

    \columnbreak

    $\begin{array}{c}
  \mathfrak R_1 = \{(-1, 0)\}\\
  \mathfrak R_2 = \{(0, -1)\}\\
  \mathfrak R_3 = \emptyset\\
  \mathfrak R_4 = \emptyset\\
  \mathfrak R^+=\{(-1, 0), (0, -1)\}
 \end{array}$

\columnbreak

\begin{center}
  Normalized action:
$
 \arraycolsep=0.0pt\begin{array}{ccc}
   x_1 &\to& x_1 + s_1 x_3x_4^2\\
   x_2 &\to& x_2 + s_2 x_3^2x_4\\
   x_3 &\to& x_3\\
    x_4 &\to& x_4\\
    \end{array}
    $
    
$(s_1, s_2) \in \Ga^2$
  \end{center}  

\end{multicols}

\end{exmp}

\begin{exmp} Let us consider the projective plane $\P^2$.
  It corresponds to the following fan with ${p_3=-p_1-p_2}$.
  This fan is not wide.
  Therefore there are two additive actions up to isomorphism.

  \begin{multicols}{3}
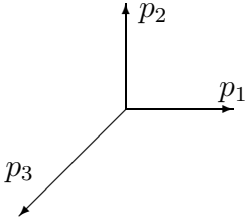
    \begin{picture}(90,90)
      \put(50,50){\vector(1,0){40}}
      \put(50,50){\vector(0,1){40}}
      \put(50,50){\vector(-1,-1){40}}      
      \put(85,55){$p_1$}
      \put(55,85){$p_2$}
      \put(5,25){$p_3$}
    \end{picture}
    
    \columnbreak

    $\arraycolsep=0.0pt
    \begin{array}{c}
  \mathfrak R_1 = \{(-1, 0), (-1, 1)\}\\
  \mathfrak R_2 = \{(0, -1), (1, -1)\}\\
  \mathfrak R_3 = \{(1, 0), (0, 1)\}\\
  \mathfrak R^+=\{(-1, 0), (0, -1), (1, -1)\}
 \end{array}$

 \columnbreak

 \begin{center}
   Normalized action:
$\arraycolsep=1.4pt
  \begin{array}{ccc}
   x_1 &\to& x_1 + s_1 x_3\\
   x_2 &\to& x_2 + s_2 x_3\\
   x_3 &\to& x_3\\
    \end{array}
    $
    
$(s_1, s_2) \in \Ga^2$
\end{center}

\begin{center}
    Non-normalized action:
$\arraycolsep=0.0pt
  \begin{array}{ccc}
   x_1 &\to& x_1 + s_1 x_3\\
   x_2 &\to& x_2 +\frac{2s_2 + s_1^2}2 x_3 + s_1 x_1\\
   x_3 &\to& x_3\\
    \end{array}
    $
    
$(s_1, s_2) \in \Ga^2$
\end{center}

\end{multicols}

\end{exmp}

\begin{exmp} Let us consider Hirzebruch surface $\mathbb F_1$.
  It corresponds to the following fan with ${p_3=-p_1-p_2, p_4=-p_2}$.
  This fan is not wide.
  Therefore there are two additive actions up to isomorphism.
    \begin{multicols}{3}
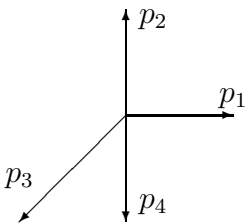
    \begin{picture}(100,100)
      \put(50,50){\vector(1,0){40}}
      \put(50,50){\vector(0,1){40}}
      \put(50,50){\vector(0,-1){40}}
      \put(50,50){\vector(-1,-1){40}}
      \put(85,55){$p_1$}
      \put(55,85){$p_2$}
      \put(55,15){$p_4$}
      \put(5,25){$p_3$}

    \end{picture}
    
\columnbreak

$
\arraycolsep=0.0pt
\begin{array}{c}
  \mathfrak R_1 = \{(-1, 0)\}\\
  \mathfrak R_2 = \{(0, -1), (1, -1)\}\\
   \mathfrak R_3 = \{(1, 0)\}\\
  \mathfrak R_4 = \emptyset\\
  \mathfrak R^+=\{(-1, 0),   (0, -1),  (1, -1)\}
 \end{array}$

\columnbreak

Normalized action:
\begin{center}
$\arraycolsep=0.pt
  \begin{array}{ccc}
   x_1 &\to& x_1 + s_1 x_3\\
   x_2 &\to& x_2 + s_2 x_3x_4\\
   x_3 &\to& x_3\\
x_4 & \to& x_4
    \end{array}
    $

     $(s_1, s_2) \in \Ga^2$
\end{center}

Non-normalized action:
\begin{center}
  $\arraycolsep=0.0pt
  \begin{array}{llc}
    x_1 &\to& x_1 + s_1 x_3\\
    x_2&\to& x_2 + \frac{2s_2+s_1^2}2x_3x_4 +s_1 x_1x_4\\
    x_3 &\to& x_3\\
    x_4 &\to& x_4
  \end{array}$

$(s_1, s_2) \in \Ga^2$
\end{center}

\end{multicols}
For a geometric realization of these two actions, see \cite[Propostion~5.5]{HT}.
\end{exmp}

Finally, let us outline some problems for further research.
\begin{prob}
Classify additive actions on complete non-toric normal surfaces.
\end{prob}
Examples of additive actions on singular del Pezzo surfaces can be found in \cite{DL}.

\bigskip

The case of 3-dimensional toric varieties seems to be more complicated: by Hassett-Tschinkel correspondence, we have four non-isomorphic additive actions on $\P^3$, see \cite[Proposition~3.3]{HT}.
Nevertheless, the following problem seems to be reasonable.

\begin{prob}
  Classify additive actions on complete three-dimensional toric varieties.
  In particular, characterize complete toric 3-folds that admit a unique additive action.
  Is it true that the number of additive actions on a complete toric 3-fold is finite?
\end{prob}


\begin{thebibliography}{99}

\bibitem{A1}
Ivan Arzhantsev. Flag varieties as equivariant compactifications of $\mathbb{G}^n_a$.
Proc. Amer. Math. Soc. 139 (2011), no.~3, 783--786
\bibitem{ABZ}
Ivan Arzhantsev, Sergey Bragin, and Yulia Zaitseva. Commutative algebraic monoid structures on affine spaces.  Comm. Contem. Math., to appear; arXiv:1809.052911
%
\bibitem{ADHL}
Ivan Arzhantsev, Ulrich Derenthal, J\"urgen Hausen, and Antonio Laface. \emph{Cox rings}.
Cambridge Studies in Adv. Math. 144, Cambridge University Press, New York, 2015
%
\bibitem{AP}
Ivan Arzhantsev and Andrey Popovskiy. Additive actions on projective hypersurfaces. In:
Automorphisms in Birational and Affine Geometry, Proc. Math. Stat. 79,
Springer, 2014, 17-33
%
\bibitem{AR}
Ivan Arzhantsev and Elena Romaskevich. Additive actions on toric varieties. Proc. Amer. Math. Soc. 145 (2017), no.~5, 1865--1879
%
\bibitem{AS}
Ivan Arzhantsev and Elena Sharoyko. Hassett-Tschinkel correspondence:
Modality and projective hypersurfaces. J. Algebra 348 (2011), no.~1, 217--232
%
\bibitem{CLT1}
Antoine Chambert-Loir and Yuri Tschinkel. On the distribution of points of bounded
height on equivariant compactifications of vector groups. Invent. Math. 148 (2002), no.~2, 421-452
%
\bibitem{CLT2}
Antoine Chambert-Loir and Yuri Tschinkel. Integral points of bounded height on partial
equivariant compactifications of vector groups. Duke Math. J. 161 (2012), no.~15, 2799--2836
%
\bibitem{Cox}
David Cox. The homogeneous coordinate ring of a toric variety. J. Alg. Geom. 4 (1995), no.~1, 17--50
%
\bibitem{CLS}
David Cox, John Little, and Henry Schenck. \emph{Toric Varieties}. Graduate Studies in Math. 124, AMS, Providence, RI, 2011
%
\bibitem{De}
Michel Demazure. Sous-groupes algebriques de rang maximum du groupe de Cremona. Ann. Sci. Ecole
Norm. Sup. 3 (1970), 507--588
%
\bibitem{DL}
Ulrich Derenthal and Daniel Loughran.
Singular del Pezzo surfaces that are equivariant compactifications.
J.~Math. Sciences 171 (2010), no.~6, 714--724
%
\bibitem{Dev}
Rostislav Devyatov. Unipotent commutative group actions on flag varieties and nilpotent
multiplications. Transform. Groups 20 (2015), no.~1, 21--64
%
\bibitem{Fe}
Evgeny Feigin. $\mathbb{G}^M_a$ degeneration of flag varieties. Selecta Math. New Ser. 18 (2012),
no.~3, 513--537
%
\bibitem{FH}
Baohua Fu and Jun-Muk Hwang. Uniqueness of equivariant compactifications of $\C^n$
by a Fano manifold of Picard number $1$. Math. Research Letters 21 (2014), no.~1, 121--125
%
\bibitem{Fu}
William Fulton. \emph{Introduction to toric varieties}. Annales of Math. Studies 131,
Princeton University Press, Princeton, NJ, 1993
%
\bibitem{HT}
Brendan Hassett and Yuri Tschinkel. Geometry of equivariant compactifications of $\mathbb{G}^n_a$.
Int. Math. Res. Notices 1999 (1999), no.~22, 1211--1230
%
\bibitem{KL}
Friedrich  Knop and Herbert Lange. Commutative algebraic groups and intersections of quadrics. Math. Ann. 267 (1984), no.~4, 555-571
%
\bibitem{Oda}
Tadao Oda. \emph{Convex bodies and algebraic geometry: an introduction to toric varieties}. A Series of Modern Surveys in Math. 15, Springer Verlag, Berlin, 1988
%

\end{thebibliography}
\end{document}